\documentclass{amsart}
\usepackage{graphicx}
\usepackage{amssymb}
\usepackage{amsfonts}
\usepackage{hyperref}
\usepackage{mathrsfs}
\usepackage[margin=2.5cm]{geometry}

  [1995/08/01 v0.1 Double stroke roman fonts]

\def\ds@whichfont{dsrom}
\relax

\DeclareMathAlphabet{\mathds}{U}{\ds@whichfont}{m}{n}

\DeclareMathOperator{\var}{Var}

\swapnumbers
\newtheorem{theorem}{Theorem}[section]
\newtheorem{lemma}[theorem]{Lemma}
\newtheorem{corollary}[theorem]{Corollary}
\newtheorem{proposition}[theorem]{Proposition}
\theoremstyle{definition}
\newtheorem{definition}[theorem]{Definition}

\newtheorem{remark}[theorem]{Remark}
\newtheorem{example}[theorem]{Example}

\numberwithin{equation}{section}
\theoremstyle{plain}

\numberwithin{equation}{section} 
\numberwithin{figure}{section} 
\theoremstyle{plain}
\theoremstyle{plain}
\theoremstyle{remark}
\newtheorem*{acknowledgement*}{Acknowledgement}

\newcommand{\cD}{{\mathcal D}}
\newcommand{\cE}{{\mathcal E}}
\newcommand{\cF}{{\mathcal F}}
\newcommand{\cG}{{\mathcal G}}
\newcommand{\cH}{{\mathcal H}}

\newcommand{\cW}{{\mathcal W}}
\newcommand{\cX}{{\mathcal X}}

\newcommand{\Om}{{\Omega}}

\newcommand{\ve}{{\varepsilon}}
\newcommand{\del}{{\delta}}

\newcommand{\sig}{{\sigma}}
\newcommand{\al}{{\alpha}}

\newcommand{\bbE}{{\mathbb E}}

\newcommand{\bbN}{{\mathbb N}}
\newcommand{\bbP}{{\mathbb P}}
\newcommand{\bbR}{{\mathbb R}}

\newcommand{\bbI}{{\mathbb I}}

\begin{document}
\title[]{On the functional CLT for slowly mixing triangular arrays}
 \author{Yeor Hafouta \\
\vskip 0.1cm
Department of Mathematics\\
The Ohio State University}
\email{yeor.hafouta@mail.huji.ac.il, hafuta.1@osu.edu}
\date{\today}
\maketitle
\markboth{Y. Hafouta}{Functional CLT}
\renewcommand{\theequation}{\arabic{section}.\arabic{equation}}
\pagenumbering{arabic}

\begin{abstract}
In \cite{MPU} a functional CLT was obtained for triangular arrays satisfying the Lindeberg condition, that the sum of the individual variances is at most the same order as the variance of the underlying sum, and under the optimal mixing rates $\sum_{n}\rho(2^n)<\infty$, where $\rho(\cdot)$ are the $\rho$-mixing coefficients of the array. In this paper we will present alternative conditions which do not involve the assumption on the sum of the variances, and instead we will assume certain maximal moment assumptions (which we can verify for $\phi$-mixing arrays) and mixing rates of the form   $\sum_n\rho(e^{G(n)})<\infty$ where $G(n)$ grows sub-linearly fast in $n$ (e.g. $G(n)=n/\ln(\ln n)$). We will also discuss alternative conditions to the ones in the functional CLT for $\al$-mixing triangular arrays which was obtained in \cite{MP}.
\end{abstract}

\section{Introduction}\label{Intro}

Let $(\xi_n)$ be a real-valued square integrable sequence of random variables and set $S_n=\sum_{j=1}^n\xi_j$. We recall that $S_n$ obeys the  central limit theorem (CLT) if $V_n=\var(S_n)\to\infty$ and $\hat S_n=(S_n-\bbE[S_n])/\sig_n$, $\sig_n=\sqrt{V_n}$ converges in distribution to the standard normal law. For stationary sequences, under  weak-dependence (mixing)  conditions, the CLT goes back to Rosenblat \cite{Rosen1956}. Since then, several variations of the CLT for weakly stationary sequences were studied in different setups by various authors, using several different methods.

For nonstationary mixing sequences, or, more generally, triangular arrays $\{\xi_{1,n},...,\xi_{n,n}\}$ the CLT itself for $S_n=\sum_{j=1}^n\xi_{j,n}$ was studied under various conditions in  \cite{Utev1991, Peligrad1996, NonStatCLT, Volny, PelBook, Rio1995, Ri, RioBook}. In comparison with the stationary case, the main difficuly in the non-stationary case is that the variance $V_n$ can diverge at an arbitrary, non-linear, rate.
 For sufficiently well contracting triangular arrays of Markov chains  the CLT for $\hat S_n$ was  established in the 50's by R. Dobrushin \cite{Dub}, and we refer to \cite{VarSeth} for a more modern presentation of Dobrushin's CLT and to \cite{Pel} for the CLT under weaker contraction conditions. 

A more refined result is the functional central limit theorem. In the stationary case it asserts that the random function $W_n(t)=\hat S_{[nt]}$ converges in distribution towards a  Brownian motion. Probably the most common method of proof is the martingale approximation technique due to Gordin \cite{Gordin}.
While the functional CLT for non-stationary martingales is well understood (see, for instance \cite{Brown, Aldous}), there are far less results about martingale approximation for nonstationary mixing sequences (or triangular arrays).
Let $V_n=\var(S_n)$ and $\sig_n=\sqrt{V_n}$.
 The first thing to observe in the non-stationary setup is that the covariance function $b_n(t,s)=\text{Cov}(\hat S_{[nt]},\hat S_{[ns]})$ of  the random function $\hat S_{[nt]}$, where $\hat S_n=\big(S_n-\bbE[S_n]\big)/\sig_n$ does not converge towards the covariance function of a Brownian motion. Hence, in the nonstationary setup the natural choice for random functions on $[0,1]$ is given by $W_n(t)=\hat S_{v_n(t)}$, where $v_n(t)=\min\{1\leq k\leq n:\,\var(S_{k,n})\geq t V_n\}$ and $S_{k,n}=\sum_{j=1}^k\xi_{j,n}$.
 
Very recently there was major progress on the functional CLT for triangular arrays. In \cite{MPU} Merlev\'ede, Peligrad and Utev answered a question raised by Ibragimov proving that the functional CLT holds if the array satisfies the classical Lindeberg condition and, in addition,
\begin{equation}\label{Cond1}
\sum_{j=1}^n\var(\xi_{j,n})=O(\var(S_n))
\end{equation}
and the $\rho$-mixing coefficients (defined by \eqref{rho def}) satisfy 
\begin{equation}\label{Cond2}
\sum_{j}\rho(2^j)<\infty.
\end{equation}
These results were obtained by a new martingale approximation technique. We remark that when $\xi_{j,n}=\xi_j$ forms a stationary sequence then under \eqref{Cond2} the limit $\lim_{n\to\infty}n^{-1}\var(S_n)=\sig^2$ exists, and so when $\sig^2>0$ the condition \eqref{Cond1} holds true since $\xi_{j,n}=\xi_j\in L^2$. Moreover, in the stationary setup, when assuming only the Lindeberg condition  the mixing rates \eqref{Cond2} are optimal even for the classical CLT to hold (see \cite{Br1980, Br1987}).

However, in the non-stationary case condition \eqref{Cond1} entails a certain ``structural" assumption, for instance, by \cite[Proposition 13]{Pel} it holds true for additive functionals $\xi_{j,n}=f_{j,n}(Y_{j,n})$ of an array of Markov chains $\{Y_{j,n}: 1\leq j\leq n\}$ whose first correlation coefficient $\rho_n(1)$ is uniformly smaller than $1$. Nevertheless, in this case $\rho(j)$ already converges exponentially fast to $0$.

The first goal of this paper is to find alternative  (to \eqref{Cond1}) sufficient conditions for the functional CLT  for (slowly) $\rho$-mixing arrays.  We will show that there is a ``trade-off" between conditions \eqref{Cond1} and \eqref{Cond2} and certain growth assumptions and sub-exponential versions of \eqref{Cond2} together with maximal moment assumptions. For instance, a particualr case of our results  is the following:
\begin{theorem}\label{IntroThm1}
Let $\{\xi_{j,n}:1\leq j\leq n\}$ be a centered triangular array which is uniformly bounded in $L^2$ so that $V_n\to\infty$.
Suppose that $\sum_{j}\rho([e^{j/\ln j}])<\infty$ and that for some $p>2$ for every $1\leq k<\ell\leq n$ we have
\begin{equation}\label{MaxMomAssIntro}
\left\|\max\{|S_{j,n}-S_{k,n}|:k\leq j\leq \ell\}\right\|_{L^p}\leq w_n+b_n\max\{\|S_{j,n}-S_{k,n}\|_{L^2}: k\leq j\leq\ell\}
\end{equation}
where  $S_{j,n}=\sum_{i=1}^k\xi_{i,n}$, $w_n=o(\sig_n)$ and $b_n=o\big((\ln\ln\sig_n)^{1/2-\del}\big)$ for some $\del\in(0,\frac12)$.
 Then $W_n(\cdot)$ converges in distribution to a standard Brownian motion. 
\end{theorem}

When  the variance of $S_{j,n}-S_{k,n}$ grows linearly fast in $j-k$ condition \eqref{MaxMomAssIntro} with $w_n=0$ and  a constant $b_n$ holds true dues to a result of Shao \cite[Theorem 1.1]{Shao1995}, under the weaker mixing condition $\sum_{j}\rho^{2/p}(2^j)<\infty$ and when $\xi_{j,n}$ are uniformly bounded in $L^p$, where $p>2$. In the stationary case by \cite[Theorem 6.16]{PelBook} we see that condition \eqref{MaxMomAss} holds true with $w_n=0$ and a constant $b_n$ if $\rho(n)\to 0$ and $\|\xi_j\|_{L^p}<\infty$ (note that in this case $\sig_n^2=nh(n)$ for a slowly varying function $h(\cdot)$). However, the main interest in this paper is the case when $\sig_n^2$ grows sub-linearly fast, and by applying
a recent   maximal moment inequality \cite[Theorem 6.17]{PelBook} established by Merlev\'ede,  Peligrad,  and Utev  for non-stationary arrays we obtain the following corollary of Theorem \ref{MaxMomAss}. 

\begin{corollary}\label{Into Cor}
Under the conditions of Theorem \ref{IntroThm1} we have the following.
Let $\phi_n(\cdot)$ be the $\phi$-mixing coefficients (defined in \eqref{phi def}) of the finite sequence $\{\xi_{j,n}: 1\leq j\leq n\}$.
 Then $W_n(\cdot)$ converges in distribution to a standard Brownian motion in the following two situations:
 \vskip0.2cm
(i) When  $\|\xi_{j,n}\|_{L^\infty}\leq K_n$ and $\phi_n(m_n)<\frac12-\ve$,
 for some $\ve>0$ and two sequences $K_n$ and $m_n$ such that $K_nm_n= o\big((\ln\ln\sig_n)^{1/2-\del}\big)$ for some $\del\in(0,\frac12)$. 
 \vskip0.1cm
(ii) When $\sig_n\geq cn^\del$ for some $\del>0$, $c>0$ and all $n$ large enough, $\|\xi_{j,n}\|_{L^p}\leq K_n$ and $\phi_n(m_n)<\frac12-\ve$ for some $p>1/\del$, $\ve>0$, $m_n$ and $K_n$ so that $m_n K_n=o(n^{\frac{\del p-1}{p}}).$
\end{corollary}
We also obtain similar results when $\sum_{j}\rho([e^{cj^\al}])<\infty$ for some $\al\in(0,1)$ and $c>0$ and $\sum_{j}\rho([j^q])<\infty$ for some $q>0$, where in these cases we can consider faster growth rates for $m_nK_n$ (in both cases (i) and (ii)).
In addition, we obtain functional CLT's  for $\al$-mixing arrays under certain type of summability assumptions, as described in the following paragraph.

In  \cite{MP}  Merlev\'ede and Peligrad applied the martingale approximation techniques developed in \cite{MPU}, and showed (in particular) that the functional CLT holds if the $\al$-mixing coefficients (defined by \eqref{al def}) satisfy $\sum_{j}j^{2/\del}\al(j)<\infty$ for some $\del>0$ so that 
\begin{equation}\label{Cond1.1}
\sum_{j=1}^n\|\xi_{j,n}-\bbE[\xi_{j,n}]\|_{L^{2+\del}}^2=O(\var(S_n)).
\end{equation}
In the stationary setup the conditions in \cite{MP} are essentially optimal, see the discussion in \cite[Section 2.1]{MP}. However, 
 as demonstrated in the applications of \cite[Corollary 2.2]{MP}, in the nonstationary case  condition   \eqref{Cond1.1} entails a certain ``structural" assumption on $\xi_{j,n}$.
The second goal of this paper is to obtain the functional CLT for $\al$-mixing arrays without  condition  \eqref{Cond1.1}.
For instance, we will show that without \eqref{Cond1.1}, under the additional assumption that $\xi_{j,n}$ are uniformly bounded in $L^q$ for some $q>2$ so that
 $\sum_{j=1}^{\infty}(\al(j))^{1/2-1/q}<\infty$,
the functional CLT holds true under the maximal moment assumption \eqref{MaxMomAssIntro}. As a corollary, it follows that the functional CLT holds true  when, instead of \eqref{MaxMomAssIntro}, we assume that $\phi_n(m_n)<\frac12-\ve$ for some $m_n$ and either $\|\xi_{j,n}\|_{L^\infty}\leq K_n$ and $m_nK_n=o(\sig_n)$ or $\sig_n\geq n^\del$ for some $\del>0$ (and all $n$ large enough) and $m_n\max_{1\leq j\leq n}\|\xi_{j,n}\|_{L^p}=o(n^{\frac{\del p-1}{p}})$ for some $p>1/\del$.

\subsection{Outline of the proof: variance regularization}\label{Reg}
We will describe the proof in the $\rho$-mixing case, and the description of the proof for $\al$-mixing arrays is similar.
Let $a_j=e^{G(j)}=e^{jo(j)}$ be a sequence which grows sub-exponentially fast (e.g. $a_j=e^{j/\ln\ln j}$) so that  $\sum_{j}\rho(a_j)<\infty$.
The strategy of the proof of Theorem \ref{IntroThm1} and its more general versions is to decompose the sum $S_n=\sum_{k=1}^{n}\xi_{k,n}$ into blocks $X_{j,n}=\sum_{k\in I_{j,n}}\xi_{k,n}$ so that $S_n=\sum_{j=1}^{u_n}X_{j,n}$,  $I_{j,n}=[x_{j,n},y_{j,n}]$ is to the left of $I_{j+1,n}$, and the blocks are ``regular" (w.r.t. $(a_j)$) in the sense that 
$$
A_1 a_j\leq \|X_{j,n}\|_{L^2}\leq\max_{k\in I_{j,n}}\left\|\sum_{\ell=x_{j,n}}^{k}\xi_{\ell,n}\right\|_{L^2}\leq A_2 a_j, 
$$
where $A_1$ and $A_2$ are positive constants 
and 
$$
\sig_n^2=\var(S_n)\asymp\sum_{j=1}^{u_n}a_j^2.
$$
In particular condition \eqref{Cond1} is valid for the new array $\{X_{j,n}:\, 1\leq j\leq u_n\}$, and 
once we show that such a decomposition exists, we would like to apply \cite[Theorem 4.1]{MPU} (discussed above). However, this requires $X_{j,n}$ to satisfy the classical Lindeberg condition. Furthermore, once the functional CLT is established for the new array $\{X_{j,n}:\,1\leq j\leq u_n\}$, we need to show that it implies the functional CLT for the original array $\{\xi_{j,n}:\,1\leq j\leq n\}$. It turns out that this task requires a  certain ``(weak) Lindeberg condition of maximal type": set $\cX_{j,n}=\max_{k\in I_{j,n}}\left|\sum_{\ell=x_{j,n}}^{k}\xi_{\ell,n}\right|$. Then the functional CLT for $\{\xi_{j,n}\}$ will follow from the functional CLT for $\{X_{j,n}\}$ if for every $\ve>0$, 
$$
\lim_{n\to\infty}\sig_n^{-1}\sum_{j=1}^{u_n}\bbE[|\cX_{j,n}|\bbI(|\cX_{j,n}|\geq \sig_n\ve)]=0.
$$
The last step would be to verify both Lindeberg conditions. This is done by a recent  maximal moment inequality  due to  Merlev\'ede,  Peligrad and Utev (\cite[Theorem 6.17]{PelBook}) which holds true in quite a general non-stationary setting, and it yields appropriate upper bounds on the norms $\|\cX_{j,n}\|_{L^p}$  under the conditions discussed in Corollary \ref{Into Cor} (that is, \cite[Theorem 6.17]{PelBook} provides sufficient conditions for \eqref{MaxMomAssIntro}).

\section{Preliminaries and  main results}
Let $\{\xi_{j,n}, 1\leq j\leq n\}$ be a triangular array of real-valued, square integrable centered random variables defined on a common probability space. Let $S_{k,n}=\sum_{j=1}^k\xi_{j,n}$, $S_n=S_{n,n}$ and set $\sig_n=\|S_n\|_{L^2}=\sqrt{\var(S_n)}$. We assume here that 
$$\lim_{n\to\infty}\sig_n=\infty$$ 
and that\footnote{This is a classical assumption which means that the individual summands are of smaller $L^2$- magnitude than the the sum itself.} $$\gamma_n:=\max_{j}\|\xi_{j,n}\|_{L^2}=o(\sig_n).$$ 

For each $t\in[0,1]$, set 
$$v_n(t)=\min\{1\leq k\leq n: \sig_{k,n}^2\geq t\sig_n^2\},$$ where $\sig_{k,n}^2$ is the variance of $\sum_{j=1}^k\xi_{j,n}$. Consider   the random function 
$$
W_n(t)=\sig_n^{-1}\sum_{j=1}^{v_n(t)}\xi_{j,n}=\sig_n^{-1}S_{v_n(t),n}
$$
on $[0,1]$. Then $W_n(\cdot)$ is a random element of the  Skorokhod space $D[0,1]$. In this paper we will prove that $W_n(\cdot)$ converges in  $D[0,1]$ towards a Brownian motion, under a variety of mixing and moment assumptions. The first thing to notice is that $W_n(\cdot)$ is left invariant after replacing $\xi_{j,n}$ by $b_n\xi_{j,n}$, where $b_n>0$ is a constant depending on $n$. Therefore, we assume here that
$$
K:=\sup_{n}\max_{j\leq n}\|\xi_{j,n}\|_{L^2}\leq 1.
$$
This assumption means that in order to translate the results presented in this paper to the more general case we need to replace $\xi_{j,n}$ with $\frac{\xi_{j,n}}{1+\gamma_n}$ and $\sig_n$ with $s_n:=\frac{\sig_n}{1+\gamma_n}\to\infty$.

\subsection{$\rho$-mixing arrays}
 Since we will have no restrictions on the joint distribution of $\xi_{j,n}$ and $\xi_{j',n'}$ for $n\not=n'$ we assume here that $\xi_{j,n}$ are defined on a joint probability space
$(\Om,\cF,\bbP)$. For any two sub-$\sig$-algebras $\cG$ and $\cH$ of $\cF$ let
\[
\rho(\cG,\cH)=\sup\{|\bbE[fg]|: f\in B_{0,2}(\cG), g\in B_{0,2}(\cH)\}
\] 
where for any sub-$\sig$-algebra $\cG$, \,$B_{0,2}(\cG)$ denotes the space of all square integrable functions $g$ so that $\bbE[g]=0$ and $\bbE[g^2]=1$. Then the $k$-th $\rho$-mixing coefficient of the array $\xi$ is defined by 
\begin{equation}\label{rho def}
\rho(k)=\sup_n\sup\{\rho(\cF_n(s),\cF_n(s+k,n)): s\leq n-k\}
\end{equation}
where $\cF_n(s)$ is the $\sig$-algebra generated by $\xi_{j,n}$ for $j\leq s$ and $\cF_n(s+k,n)$ is the $\sig$-algebra generated by $\xi_{j,n}$ for $s+k\leq j\leq n$. Let us define $\rho(x)=\rho([x])$ for all $x>0$.


\subsection{Subexponential seuqneces and regular blocks: the general result}
Theorem \ref{IntroThm1} and Corollary \ref{Into Cor} will follow from the results presented in this section (Theorem \ref{Cor2.7} and Corollary \ref{Cor2.9}, respectively). Let us first start with the preparation for the construction of the regular blocks, as described in Section \ref{Reg}.

We consider here the following two properties, which for the sake of convenience are presented as definitions.
\begin{definition}\label{Def1}
We say that a sequence $(a_j)$ diverges sub-exponentially fast in a good way if  $a_j\asymp e^{G(j)}$ for some positive differentiable function $G$ on $[1,\infty)$ so that:

 \begin{enumerate}
 \item 
$\lim_{x\to\infty}G(x)=\infty$;
\vskip0.1cm
\item  there are constants $C_1,C_2,c>0$ and a positive function $H:[1,\infty)\to\bbR$ so that $\lim_{u\to\infty}H(u)=0$ and for every $u\geq c$ we have 
 $$
 C_1H(u)\leq G'(x)\leq C_2H(u);$$
  \vskip0.1cm
\item there are  constants $\ve,\del\in(0,1)$ so that for all $u$ large enough we have $$e^{G(u)}>(1-\ve)e^{G(u(1-\del))}.$$
 \end{enumerate}

\end{definition}

\begin{example}[Main examples]\label{Eg G}
\,
\begin{enumerate}
\item $a_j=j^q$, with $H(u)\asymp \frac1 u$.
\vskip 0.1cm
\item $a_j=e^{(\ln j)^s}$ for some $s>1$, with $H(u)\asymp \frac{(\ln u)^{s-1}}{u}$.
\vskip0.1cm
\item
$a_j=e^{cj^\al}$ where $c>0$ and $\al\in(0,1)$, with $H(u)\asymp u^{\al-1}$.
\vskip0.1cm
\item
$a_j=e^{jg(j)}$, where $g(j)=\frac 1{\ln^{d,\circ}(j)}$ for some $d\geq1$, where $f^{d,\circ}=f\circ f\circ f\circ...\circ f$ ($d$ times). In this case $H(u)\asymp\frac{1}{\ln^{d,\circ} (u)}$.
\end{enumerate}
\end{example}

\begin{remark}
Our main theorems will hold under the assumption that $\sum_j\rho(a_j)<\infty$. This condition holds when:
\begin{enumerate}
\item  $a_j\asymp j^q$ for some $q>0$ snd $\rho(n)\leq n^{-\ve}$ for some arbitrary small $\ve>0$.
\vskip0.1cm
\item $a_j\asymp e^{(\ln j)^s}$ for some $s>1$ and $\rho(n)\leq e^{-(1+\ve)(\ln n)^{1/s}}$ for some $s>1$ and $\ve>0$.
\vskip0.1cm
\item  $a_j\asymp e^{cj^\al}$ for some $c>0$ and $\al\in(0,1)$ and $\rho(n)\leq (\ln n)^{-s}$ for some $s>1$.
\vskip0.1cm
\item $a_j$ is like in Example \ref{Eg G} (4) and $\rho(n)\leq \frac{1}{\ln n(\ln\ln n)^s}$ for some $s>2$.
\end{enumerate} 
\end{remark}
The proof of Theorem \ref{IntroThm1} and \ref{Into Cor} and the more general results which will be proved in this paper are based on the following idea.

\begin{definition}[Regular blocks]
Given a sequence $(a_j)$ of positive numbers and two positive constants $C_1<C_2$,
we say that a family of partitions of $\{1,2,...,n\}$ into intervals (blocks) in the positive integers $$I_{j,n}=[x_{j,n},y_{j,n}]\cap \bbN,\, j\leq u_n$$ 
 is \textit{regular} (w.r.t $\{\xi_{j,n}\}$, $(a_j)$, $C_1$ and $C_2$) if
\begin{enumerate}
\item $I_{j,n}$ is to the left of $I_{j+1,n}$;
\vskip0.2cm 
\item for every $1\leq s_1<s_2\leq u_n$, with  $X_{j,n}=\sum_{k\in I_{j,n}}\xi_{k,n}$, we have
$$
C_1a_j\leq \|X_{j,n}\|_{L^2}\leq\max_{k\in I_{j,n}}\left\|\sum_{\ell=x_{j,n}}^k\xi_{\ell,n}\right\|_{L^2}\leq C_2a_j;
$$
\vskip0.2cm
\item $$
C_1\sum_{j=s_1}^{s_2}\var(X_{j,n})\leq \var\left(\sum_{j=s_1}^{s_2}X_{j,n}\right)\leq C_2\sum_{j=s_1}^{s_2}\var(X_{j,n}).
$$
\end{enumerate}

Blocks $X_{j,n}$ which are constructed in this manner will be refereed to as \textit{regular blocks} (w.r.t. $\{\xi_{j,n}\}, (a_j), C_1$ and $C_2$).
\vskip0.1cm
Set $\cX_{j,n}=\max\{|\sum_{s=x_{j,n}}^{k}\xi_{s,n}|: k\in I_{j,n}\}$. 
We say that the regular blocks $\{X_{j,n}\}$ satisfy a Lindeberg condition of maximal type if  every
 $\ve>0$,
$$
\lim_{n\to\infty}\sig_n^{-2}\sum_{j=1}^{u_n}\bbE[|\cX_{j,n}|^2\bbI(|\cX_{j,n}|\geq \ve\sig_n)]=0.
$$
\end{definition}

\begin{remark}\label{R1}
Since $\sig_n^2\asymp \sum_{j=1}^{u_n}a_j^2$, when $a_j\asymp j^q$ we get $u_n\asymp \sig_n^{\frac{2}{2q+1}}$. In all the other three special cases discussed in Example \ref{Eg G}, we have 
$$\sig_n^2\asymp\int_{(1-\del)u_n}^{u_n}e^{2G(x)}=\int_{(1-\del)u_n}^{u_n}\left(e^{2G(x)}G'(x)\right)/G'(x)\asymp H(u_n)^{-1}e^{2G(u_n)}.$$ 
Thus:  
\begin{enumerate}
\item 
when $a_j\asymp e^{(\ln j)^s}$ then $u_n\asymp e^{(\ln\sig_n)^{1/s}}$;
\vskip0.1cm
\item
 when $a_j\asymp e^{cj^\al}$ then $u_n\asymp (\ln \sig_n)^{1/\al}$; 
\vskip0.1cm
 \item
 when $a_j\asymp e^{j/\ln^{d,\circ}(j)}$  then $u_n$ is at least of order $\ln(\sig_n)$ and at most of order $\ln(\sig_n)\ln^{d,\circ}(\sig_n)$.

\end{enumerate} 
\end{remark}

Our first result for $\rho$-mixing arrays is as follows.

\begin{theorem}\label{Rho}
Suppose that $\sig_n\to\infty$.
Let $a_j=e^{G(j)}$  diverge sub-exponentially fast in a good way. Suppose that 
$\sum_{j}\rho(a_j)<\infty$.
 Then there are absolute constants $C_1,C_2>0$ for which one can construct regular blocks $\{X_{j,n}:\, 1\leq j\leq u_n\}$  w.r.t. $\{\xi_{j,n}\}, (a_j), C_1$ and $C_2$. Moreover,  if the maximal Lindeberg condition holds along the blocks $\{X_{j,n}\}$, then $W_n(t)$ converges in distribution towards a standard Brownian motion.
\end{theorem}

\subsubsection{\textbf{Main applications: verification of the maximal Lindeberg condition under maximal moment assumptions}}
In this section we will formulate our main applications of Theorem \ref{Rho}. We refer to Section \ref{Intro} for a discussion and comparisons with \cite[Theorem 4.1]{MPU}.

Let us fix some subexponential seqeunce $(a_j)$ as described in Definition \ref{Def1}. The following result is a corollary of Theorem \ref{Rho}, and it is a more general version of Theorem \ref{IntroThm1}.
\begin{theorem}\label{Cor2.7}
Let $\{\xi_{j,n}:1\leq j\leq n\}$ be a centered triangular array which is uniformly bounded in $L^2$.
Assume also that $\sig_n\to\infty$.
Suppose that $\sum_{j}\rho(a_j)<\infty$ and that for some $p>2$ for every $1\leq k<\ell\leq n$ the following maximal inequality holds true:
\begin{equation}\label{MaxMomAss}
\left\|\max\{|S_{j,n}-S_{k,n}|:k\leq j\leq \ell\}\right\|_{L^p}\leq w_n+b_n\max\{\|S_{j,n}-S_{k,n}\|_{L^2}: k\leq j\leq\ell\}
\end{equation}
where $w_n=o(\sig_n)$ and $b_n=o\big(H(u_n)^{p/2-1})$.
 Then the maximal Lindeberg condition holds true, and so $W_n(\cdot)$ converges in distribution to a standard Brownian motion. 
\end{theorem}

\subsubsection*{\textbf{On the verification of the maximal moment condition using the $\phi$-mixing coefficients}}
For any two sub-$\sig$-algebras $\cG$ and $\cH$ of $\cF$ we set
$$
\phi(\cG,\cH)=\sup\{|\bbP(B|A)-\bbP(B)|: A\in\cG, B\in\cH, \bbP(A)>0\}.
$$
Then for a fixed $n$, the $\phi$-mixing coefficients of the finite sequence $\{\xi_{j,n}: 1\leq j\leq n\}$ are given by 
\begin{equation}\label{phi def}
\phi_n(k)=\sup\{\phi_n(\cF_{n}(s), \cF_n(s+k,n)): s\leq n-k\}.
\end{equation}
where $\cF_n(s)$ and $\cF_n(s+k,n)$ were defined after \eqref{rho def}. 
We assume here that for all $n$ large enough we have
\begin{equation}\label{jn ass}
\phi_n(m_n)<\frac12-\ve\, \text{ for some }\,m_n<n\, \text{ and a constant }\,\ve>0.
\end{equation}

\begin{remark}
Let $\phi(k)=\sup_n\phi_n(k)$ be the $\phi$-mixing coefficients of the array $\xi$. Then by \cite[Lemma 1.17]{IF},
\begin{equation}\label{rho phi}
\rho(j)\leq 2\sqrt{\phi(j)}.
\end{equation}
\end{remark} 

The following result will follow from Theorem \ref{Cor2.7}, and it is  a  more general version of Corollary \ref{Into Cor}.
\begin{corollary}\label{Cor2.9}
Under the assumptions of Corollary \ref{Cor2.7} the functional CLT holds true in the following two cases:

 \vskip0.2cm
(i) When  $\|\xi_{j,n}\|_{L^\infty}\leq K_n$ and $\phi_n(m_n)<\frac12-\ve$,
 for some $\ve>0$ and two sequences $K_n$ and $m_n$ such that $K_nm_n=o(H(u_n)^{p/2-1})$.
 \vskip0.1cm
(ii) When $\sig_n\geq cn^\del$ for some $\del>0$, $c>0$ and all $n$ large enough, $\|\xi_{j,n}\|_{L^p}\leq K_n$ and $\phi_n(m_n)<\frac12-\ve$ for some $p>1/\del$, $\ve>0$, $m_n$ and $K_n$ so that $m_n K_n=o(n^{\frac{\del p-1}{p}}).$
\end{corollary}
Using \eqref{rho phi}, we see that the functional CLT holds true when the mixing rates in Theorem \ref{Rho}  hold true for $\sqrt{\phi(\cdot)}$ in place of $\rho(\cdot)$ (which, in particular, implies \eqref{jn ass}) and the (moment) conditions in (i) or (ii) above hold true with $K_n$ so that either $K_n=o(H(u_n)^{p/2-1})$ or $K_n=o(n^{\frac{\del p-1}{p}})$ (indeed in this case $m_n$ can be replaced with  a constant).

\subsection{$\al$ mixing arrays}
 For any two sub-$\sig$-algebras $\cG$ and $\cH$ of $\cF$ let
\[
\al(\cG,\cH)=\sup\{|\bbP(A\cap B)-\bbP(A)\bbP(B)|: A\in\cG, B\in\cH\}.
\] 
 Then the $k$-th $\al$-mixing coefficient of the array $\xi$ is defined by 
\begin{equation}\label{al def}
\al(k)=\sup_n\sup\{\al(\cF_n(s),\cF_n(s+k,n)): s\leq n-k\}
\end{equation}
where $\cF_n(s)$ and $\cF_n(s+k,n)$ are defined after \eqref{rho def}.

Our main result for $\al$-mixing arrays is as follows: 
\begin{theorem}\label{Alpha}
Let 
$$
\del_n(m)=\sup_{k}\sum_{s=m}^{n-k}\left\|\bbE[\xi_{k+s,n}|\xi_{1,n},...,\xi_{k,n}]\right\|_{L^2}
$$
and suppose that $\del_n(r)<\frac14$ for some $r$ and all  $n$ large enough.
 Then there are absolute constants $C_1,C_2$ for which it is possible to construct regular blocks $\{X_{j,n}\}$ corresponding to any constant sequence $a_j=a$ with  $a>0$ large enough. Moreover, the functional  CLT holds true under the following two additional conditions:
 \begin{enumerate}
 \item the maximal Lindeberg condition holds along the blocks $X_{j,n}$;
 \vskip0.2cm
 \item  for some  $\del>0$ we have 
 \begin{equation}\label{Scond1}
\sup_{j,n}\|X_{j,n}\|_{L^{2+\del}}<\infty
 \end{equation}
and
 $\sum_{j}n^{2/\del}\al(j)<\infty$. 
 \end{enumerate}

\end{theorem}

\begin{remark}
In Lemma \ref{al lem} we show that the condition that  $\sup_n\del_n(r)<\frac14$ for some $r$ holds true when  $\xi_{j,n}$ are uniformly bounded in $L^q$ for some $q>2$ so that
 $\sum_{j=1}^{\infty}(\al(j))^{1/2-1/q}<\infty$.
 \end{remark}

\begin{corollary}\label{al cor}
Under the conditions of Theorem \ref{Alpha}, the maximal Lindeberg condition (and hence the functional CLT) holds true under \eqref{MaxMomAss}. Hence the functional CLT holds true 
 when $\phi_n(m_n)<\frac12-\ve$ for some $m_n$ and either $\|\xi_{j,n}\|_{L^\infty}\leq K_n$ and $m_nK_n=o(\sig_n)$ or $\sig_n\geq cn^\del$ for some $\del,c>0$ (and all $n$ large enough) and $m_n\max_{1\leq j\leq n}\|\xi_{j,n}\|_{L^p}=o(n^{\frac{\del p-1}{p}})$ for some $p>1/\del$.
\end{corollary}
 We refer to Section \ref{Intro} for a discussion and comparison with the functional CLT for $\al$-mixing arrays in \cite{MP}.

\section{Regular blocks }


\subsection{Construction of regular blocks using $\rho$-mixing coefficients}\label{Sec Block 1}
In this section we fix $n$ and omit the subscript $n$, and just write $\xi_{j,n}=\xi_j$.
We also set 
$$
S_{k}=\sum_{j=1}^k\xi_j
$$
and for each $B\subset\{1,2,...,n\}$ we write 
$$
S(B)=\sum_{j\in B}\xi_{j}.
$$
Let $(a_j)_{j\geq 1}$ be a sequence of positive numbers so that $a_j\geq 3$. 
\vskip0.2cm
\paragraph{\textbf{The blocks: }}
Let us fix some $A>1$ (which does not depend on $n$). Let $b_1$ be the first index $b$ so that $\|S_{b_1}\|_{L^2}\geq Aa_1$. Then 
$$A\|S_{b_1}\|_{L^2}\leq Aa_1+1$$
where we have used our assumption that $\|\xi_{j}\|_{L^2}\leq1$. 
Next, let $b_2$ be the first index $b$ so that $b\geq b_1+a_j$ and $\|S_{b_1}\|_{L^2}\geq Aa_2$. Then $$\|S_{b_1}\|_{L^2}\leq Aa_2+1.$$
 Continuing this way we get blocks of the form $B_j=B_{j,n}=\{b_j,b_{j}+1,...,b_{j}+d_j-1\}$, $j\leq u_n$ so that
 \begin{enumerate}
 \item  $Aa_j\leq \|S(B_j)\|_{L^2}\leq Aa_{j}+1$;
 \vskip0.1cm
 \item $B_j$ is to the left of $B_{j+1}$;
 \vskip0.1cm
 \item the distance between $B_j$ and $B_{j+1}$ is $a_j$.
\end{enumerate}  
Let us also add $\{b_{u_n}+d_{u_n},...,n\}$ to the last block, and let us denote by $D_j$ the gap (interval) between $B_j$ and $B_{j+1}$. Set 
$$Y_j=S(B_j),\, Z_j=S(D_j),\, X_j=Y_j+Z_j.$$
 Then $\|Z_j\|_{L^2}\leq a_j$ and $$S_n=\sum_{j=1}^n\xi_j=\sum_{j=1}^{u_n}X_j.$$
\begin{remark}
$B_j, D_j,Y_j, Z_j$ and $X_j$ depend on $n$. In the next sections we will denote them  by $B_{j,n}, D_{j,n}, Y_{j,n}, Z_{j,n}$ and $X_{j,n}$, respectively, but in order not to overload the notations in this section we will suppress the dependence on $n$ and write $X_{j,n}=X_j$ etc.
\end{remark}

Our main result here is the following proposition, which shows that $X_{j,n}$ has the properties described in Theorem \ref{Rho} (i.e. they are ``regular blocks" w.r.t. some constants and the sequence $(a_j)$.)
\begin{proposition}\label{Prop}
Suppose that $\sum_j\rho(a_j)\leq \frac14$ and that $A$ is large enough so that
$$
\left(\frac{3}{A^2}+2\sqrt{\frac{3}{A^2}}\right)\leq \frac12.
$$
Then for every $1\leq s_1<s_2\leq n$ we have 
$$
\frac{5}{24}\sum_{j=s_1}^{s_2}\var(X_j)\leq \var\left(\sum_{j=s_1}^{s_2}X_j\right)\leq \frac{7}{2}\sum_{j=s_1}^{s_2}\var(X_j).
$$
\end{proposition}
The first part of the proof is the following lemma.
\begin{lemma}\label{SumVars}
Assume that $\sum_j\rho(a_j)\leq \frac14$. Then,
for all $1\leq s_1<s_2\leq u_n$ we have
$$
\frac12\sum_{j=s_1}^{s_2}\var(Y_j)\leq \var\left(\sum_{j=s_1}^{s_2}Y_j\right)\leq \frac32\sum_{j=s_1}^{s_2}\var(Y_j)
$$
and 
$$
\frac12\sum_{j=s_1}^{s_2}\var(Z_j)\leq \var\left(\sum_{j=s_1}^{s_2}Z_j\right)\leq \frac32\sum_{j=s_1}^{s_2}\var(Z_j).
$$
\end{lemma}

\begin{proof}
Let us prove the first estimate. First,
$$
\var\left(\sum_{j=s_1}^{s_2}Y_j\right)=\sum_{i=s_1}^{s_2}\|Y_i\|_{L^2}^2+2\sum_{s_1\leq i<j\leq s_2}\text{Cov}(Y_i,Y_j).
$$
Since the size of the gap between $B_i$ and $B_{i+1}$ is $a_i$, with $A_{i,j}=a_i+...+a_{j-1}\geq \max(a_i,a_{j-1})$ we have
\begin{equation}\label{SimTo}
2\sum_{s_1\leq i<j\leq s_2}|\text{Cov}(Y_i,Y_j)|\leq 2\sum_{s_1\leq i<j\leq s_2}\rho(A_{i,j})
\|Y_i\|_{L^2}\|Y_j\|_{L^2}
\end{equation}
$$\leq  \sum_{s_1\leq i<j\leq s_2}\rho(A_{i,j})(\|Y_i\|_{L^2}^2+\|Y_j\|_{L^2}^2)
= \sum_{j=s_1}^{s_2}\|Y_j\|_{L^2}^2\sum_{i=s_1}^{j-1}\rho(A_{i,j})+
$$
$$
\sum_{i=s_1}^{s_2-1}\|Y_i\|_{L^2}^2\sum_{j=i+1}^{s_2}\rho(A_{i,j})
\leq\frac12\sum_{j=s_1}^{s_2}\|Y_j\|_{L^2}^2.
$$

To prove the corresponding estimate with $Z_j$ instead of $Y_j$, we observe that the cardinality $|B_j|$ satisfies
$|B_j|\geq\|Y_j\|_{L^2}\geq Aa_j$, and so $|B_j|\geq Aa_j+1\geq a_j$. Hence the size of the gap between $Z_j$ and $Z_{j+1}$ is at least $a_j$, and the proof proceeds similarly to the above.
\end{proof}

Next, we also need:
\begin{lemma}\label{2Lemm}
Assume that $\sum_j\rho(a_j)\leq \frac14$.  Then
for all $1\leq s_1<s_2\leq u_n$ we have 
$$
\left|\frac{\var(\sum_{j=s_1}^{s_2}X_j)}{\var(\sum_{j=s_1}^{s_2}Y_j)}-1\right|\leq \cE(A):=
\left(\frac{3}{A^2}+2\sqrt{\frac{3}{A^2}}\right).
$$
\end{lemma}
\begin{proof}
We have 
$$
\var\left(\sum_{j=s_1}^{s_2}X_j\right)=\var\left(\sum_{j=s_1}^{s_2}Y_j\right)+\var\left(\sum_{j=s_1}^{s_2}Z_j\right)+2\text{Cov}\left(\sum_{j=s_1}^{s_2}Y_j,\sum_{j=s_1}^{s_2}Z_j\right).
$$
Observe that $\|Z_j\|_2\leq a_j\leq \|Y_j\|_2/A$, and so 
by Lemma \ref{SumVars} we have 
$$
\var\left(\sum_{j=s_1}^{s_2}Z_j\right)\leq \frac32\sum_{j=s_1}^{s_2}\var(Z_j)\leq 
\frac3{2A}\sum_{j=s_1}^{s_2}\var(Y_j)\leq \frac{3}{A^2}\var\left(\sum_{j=s_1}^{s_2}Y_j\right).
$$
Thus, using also the Cauchy-Schwartz inequality we have
$$
\left|\var\left(\sum_{j=s_1}^{s_2}X_j\right)-\var\left(\sum_{j=s_1}^{s_2}Y_j\right)\right|
\leq \frac{3}{A^2}\var\left(\sum_{j=s_1}^{s_2}Y_j\right)+2\left\|\sum_{j=s_1}^{s_2}Z_j\right\|_{L^2}\left\|\sum_{j=s_1}^{s_2}Y_j\right\|_{L^2}
$$
$$
\leq \left(\frac{3}{A^2}+2\sqrt{\frac{3}{A^2}}\right)\var\left(\sum_{j=s_1}^{s_2}Y_j\right).
$$
\end{proof}

\begin{proof}[Proof of Proposition \ref{Prop}]
We need to show that 
if  $\cE(A)\leq \frac12$ then
$$
\frac{5}{24}\sum_{j=s_1}^{s_2}\var(X_j)\leq \var\left(\sum_{j=s_1}^{s_2}X_j\right)\leq\frac{21}{6}\sum_{j=s_1}^{s_2}\var(X_j).
$$

We first recall that
$$
\|Z_j\|_{L^2}\leq a_j\leq \|Y_j\|_{L^2}/A
$$
and therefore 
$$
\left(1-A^{-1}\right)\|Y_j\|_{L^2}\leq \|X_j\|_{L^2}=\|Y_j+Z_j\|_{L^2}\leq\left(1+A^{-1}\right)\|Y_j\|_{L^2}.
$$
We also note that $A^{-1}\leq \frac16$ since $\cE(A)\leq \frac12$.
Now Proposition \ref{Prop} follows by Lemma \ref{2Lemm} and the above estimates. 
\end{proof}

\subsection{Construction of regular blocks with bounded variances using  projective conditions with applications to $\alpha$-mixing arrays}\label{Al Sec}

Let us  set
$$
\del_n(m)=\sup_{k}\sum_{s=m}^{n-k}\left\|\bbE[\xi_{k+s,n}|\xi_{1,n},...,\xi_{k,n}]\right\|_{L^2}.
$$
\begin{lemma}
Suppose that $\sup_{j,n}\|\xi_{j,n}\|_{L^q}\leq A_q<\infty$ for some $q>2$. Then there is a constant $C_q$ which depends only on $q$ so that
$$
\del_n(m)\leq C_qA_q\sum_{j=m}^{\infty}(\al(j))^{1/2-1/q}.
$$
\end{lemma}\label{al lem}
\begin{proof}
For each $q,p\geq1$ and two sub-$\sig$-algebras $\cG,\cH$ let 
$$
\varpi_{q,p}(\cG,\cH)=\sup\{\|\bbE[h|\cG]-\bbE[h]\|_{L^p}: h\in L^\infty(\Om,\cH,\bbP), \|h\|_{L^q}\leq1\}.
$$
Then  (see \cite[Ch. 4]{Br}),
$$
\al(\cG,\cH)=\frac14\varpi_{\infty,1}(\cG,\cH).
$$
Thus, by applying the Riesz–Thorin interpolation theorem (\cite[Ch.6]{Folland}) with the operator $h\to\bbE[h|\cG]-\bbE[h]$ we see that
there is a constant $C_q>0$ so that
$$
\left\|\bbE[\xi_{k+s,n}|\xi_{1,n},...,\xi_{k,n}]\right\|_{L^2}\leq A_q\varpi_{q,2}(\cG,\cH)\leq A_qC_q\big(\al(\cG,\cH)\big)^{1/2-1/q}\leq A_qC_q \big(\al(s)\big)^{1/2-1/q}
$$
where $\cH$ is the $\sig$-algebra generated by $\xi_{k+s,n}$ and $\cG$ is the $\sig$-algebra generated by $\{\xi_{1,n},...,\xi_{k,n}\}$.
\end{proof}
Let  $r$ be a number which does not depend on $n$ so that 
$$
\del_n(r)<\frac14.
$$
\subsubsection*{\textbf{The blocks}}
The construction of the blocks using the coefficient $\del_n(r)$ is similar to the $\rho$-mixing case, with a small difference. We will need also to control (from below) the size gap between two consecutive ``big" and small blocks.

Let us fix some $A>r$ and let $\ve\in(0,1)$  be so that $A\ve>r$ (both $A$ and $\ve$ do not depend on $n$). Let us take $b_1$ to be the first time that 
$$\|S_{b_1}\|_{L^2}\geq A.$$ Set $Y_1=S_{b_1}$. Now we take $\beta_1$ to be the smallest positive integer so that 
$$\|S_{b_1+\beta_1}-S_{b_1}\|_{L^2}\geq A\ve.$$ 
Set $Z_1=S_{b_1+\beta_1}-S_{b_1}$. Continuing that way we get blocks $Y_1,Z_1,Y_2,Z_2,...$ of the form $\sum_{j\in I}\xi_{j,n}$ for an interval $I$ so that 
\begin{enumerate}
\item the size of the gap between two consecutive $Y_j$'s is at least $A$;
\vskip0.1cm
\item  the size of the gap  between  between two consecutive $Z_j$'s is at least $A\ve$;
\vskip0.1cm
\item $$A\leq \|Y_j\|_{L^2}\leq (A+1),\,\,A\ve \leq \|Z_j\|_{L^2}\leq A\ve+1.$$
\end{enumerate}
\vskip0.1cm
Of course, $Y_j$ and $Z_j$ depend on $n$, and when needed we will express this dependence by writing $Y_{j,n}=Y_j$ etc. 
As in Section \ref{Sec Block 1},
let us define $X_j=X_{j,n}=Y_{j,n}+Z_{j,n}$.
We prove here the following version of Proposition \ref{Prop}.

\begin{proposition}\label{Prop.1}
If 
$$
 \left(6\ve^2+2\sqrt{6}\ve\right)\leq \frac12,
$$
then for every $1\leq s_1<s_2\leq u_n$ we have 
$$
\frac{5}{24}\sum_{j=s_1}^{s_2}\var(X_j)\leq \var\left(\sum_{j=s_1}^{s_2}X_j\right)\leq \frac{7}{2}\sum_{j=s_1}^{s_2}\var(X_j).
$$
\end{proposition}

The first part of the proof is the following version of Lemma \ref{SumVars}.
\begin{lemma}\label{SumVars.1}
We have 
$$
\frac12\sum_{j=s_1}^{s_2}\var(Y_j)\leq \var\left(\sum_{j=s_1}^{s_2}Y_j\right)\leq \frac32\sum_{j=s_1}^{s_2}\var(Y_j)
$$
and 
$$
\frac12\sum_{j=s_1}^{s_2}\var(Z_j)\leq \var\left(\sum_{j=s_1}^{s_2}Z_j\right)\leq \frac32\sum_{j=s_1}^{s_2}\var(Z_j).
$$
\end{lemma}

\begin{proof}
Let us prove the first estimate. First,
$$
\var\left(\sum_{j=s_1}^{s_2}Y_j\right)=\sum_{i=s_1}^{s_2}\|Y_i\|_{L^2}^2+2\sum_{s_1\leq i<j\leq s_2}\text{Cov}(Y_i,Y_j).
$$
Since the size of the gap between $Y_i$ and $Y_{i+1}$ is at least $r$ and $\|Y_i\|_{2}\geq 1$, we have
\begin{equation}\label{SimTo.1}
2\left|\sum_{s_1\leq i<j\leq s_2}\text{Cov}(Y_i,Y_j)\right|\leq 2\sum_{s_1\leq i<s_2}\left|\bbE[Y_i\sum_{j>i}Y_j]\right|
\end{equation}
$$
=2\sum_{s_1\leq i<s_2}\left|\bbE\left[Y_i\bbE\big[\sum_{j>i}Y_j|Y_i\big]\right]\right|\leq 
2\sum_{s_1\leq i<s_2}\|Y_i\|_2\left\|\bbE\big[\sum_{j>i}Y_j|Y_i\big]\right\|_2$$$$
\leq 2\del_n(r)\sum_{s_1\leq i<s_2}\|Y_i\|_2\leq \frac12\sum_{s_1\leq i<s_2}\|Y_i\|_2^2.
$$
The proof for the $Z_j$'s is similar.
\end{proof}

We next need the following version of Lemma \ref{2Lemm}.
\begin{lemma}\label{2Lemm.1}
For all $1\leq s_1<s_2\leq u_n$ we have 
$$
\left|\frac{\var(\sum_{j=s_1}^{s_2}X_j)}{\var(\sum_{j=s_1}^{s_2}Y_j)}-1\right|\leq \cD(\ve):=
 \left(6\ve^2+2\sqrt{6}\ve\right).
$$
\end{lemma}
\begin{proof}
We have 
$$
\var\left(\sum_{j=s_1}^{s_2}X_j\right)=\var\left(\sum_{j=s_1}^{s_2}Y_j\right)+\var\left(\sum_{j=s_1}^{s_2}Z_j\right)+2\text{Cov}\left(\sum_{j=s_1}^{s_2}Y_j,\sum_{j=s_1}^{s_2}Z_j\right).
$$
Observe that $\|Z_j\|_2\leq (A\ve+1)\leq 2A\ve\leq 2\ve\|Y_j\|_2$, and so 
by Lemma \ref{SumVars.1} we have 
$$
\var\left(\sum_{j=s_1}^{s_2}Z_j\right)\leq \frac32\sum_{j=s_1}^{s_2}\var(Z_j)\leq 
\frac{3\cdot 4\ve^2}{2}\sum_{j=s_1}^{s_2}\var(Y_j)\leq (6\ve^2)\var\left(\sum_{j=s_1}^{s_2}Y_j\right).
$$
Thus, using also the Cauchy-Schwartz inequality we get that
$$
\left|\var\left(\sum_{j=s_1}^{s_2}X_j\right)-\var\left(\sum_{j=s_1}^{s_2}Y_j\right)\right|
\leq (6\ve^2)\var\left(\sum_{j=s_1}^{s_2}Y_j\right)+2\left\|\sum_{j=s_1}^{s_2}Z_j\right\|_{L^2}\left\|\sum_{j=s_1}^{s_2}Y_j\right\|_{L^2}
$$
$$
\leq \left(6\ve^2+2\sqrt{6}\ve\right)\var\left(\sum_{j=s_1}^{s_2}Y_j\right).
$$
\end{proof}

The proof of Proposition \ref{Prop.1} is now completed similarly to the proof of Proposition \ref{Prop}.

\section{Proofs of the functional CLTs}

\subsection{Proof of Theorem \ref{Rho}}
Let  $\sig_{k,n}^2$ denote the variance of $\sum_{j=1}^k\xi_{j,n}$. Let also denote by $s_{k,n}^2$ the variance of $\sum_{j=1}^kX_{j,n}$. Then $s_{u_n,u}^2=\sig_{n,n}^2=\var(S_n)=\sig_n^2$. Set 
$$v_n(t)=\inf\{1\leq k\leq n:\,\sig_{k,n}^2\geq \sig_n^2t\},\,
\tilde v_n(t)=\inf\{1\leq k\leq u_n:\,s_{k,n}^2\geq \sig_n^2t\}.
$$
Set also
$$
W_n(t)=\sig_n^{-1}\sum_{j=1}^{v_n(t)}\xi_{j,n},\,\, \cW_n(t)=\sig_n^{-1}\sum_{j=1}^{\tilde v_n(t)}X_{j,n}.
$$

\begin{lemma}\label{L1}
Suppose that the maximal Lindeberg condition holds true along the blocks $X_{j,n}$.
Then 
$$
\lim_{n\to\infty}\left\|\sup_{t\in[0,1]}|W_n(t)-\cW_n(t)|\right\|_{L^1}=0.
$$
\end{lemma}
\begin{proof}
Let $I_{j,n}=[x_{j,n},y_{j,n}]$  be so that $X_{j,n}=\sum_{k\in I_{j,n}}\xi_{k,n}$. Let 
$$
\cX_{j,n}=\max\left\{\left|\sum_{k=x_{j,n}}^{m}\xi_{k,n}\right|: x_{j,n}\leq m\leq y_{j,n}\right\}.
$$
By considering the block $I_{j,n}$ so that $v_n(t)\in I_{j,n}$ we see that 
$$
\sup_{t\in[0,1]}|W_n(t)-\cW_n(t)|\leq\max\{\cX_{j,n}/\sig_n: 1\leq j\leq u_n\}. 
$$
Now, for each $\ve>0$, 
$$
\bbE[\max\{\cX_{j,n}/\sig_n: 1\leq j\leq u_n\}]\leq\ve+\sig_n^{-1}\sum_{j=1}^{u_n}\bbE[\cX_{j,n}\bbI(\cX_{j,n}\geq \ve \sig_n)]. 
$$
Because of the maximal Lindeberg condition, for each fixed $\ve$ the second summand in the above right hand side converges to $0$ as $n\to\infty$.
\end{proof}

\begin{proof}[Proof of Theorem \ref{Rho}]
The construction of the regular block was already done in Section \ref{Sec Block 1}.
To prove the functional CLT, notice that under the assumptions of Theorem \ref{Rho}
 the array $\{X_{j,n}:\, 1\leq j\leq u_n\}$ satisfies the usual Lindeberg condition. Thus,
by applying \cite[Theorem 4.1]{MPU} we see that $\cW_n(\cdot)$ converges in distribution towards a standard Brownian motion. Now Theorem \ref{Rho} follows from Lemma  \ref{L1}.
\end{proof}

\subsection{Proof of Theorem \ref{Cor2.7}}

To prove Theorem \ref{Cor2.7},  using that $\sig_n^2\asymp\sum_{j=1}^{u_n}a_j^2$,  standard arguments involving the H\"older and the Markov inequalities show that the maximal Lindeberng condition holds true if for some $p>2$ we have
\begin{equation}\label{C}
\lim_{n\to\infty}\frac{\sum_{j=1}^{u_n}\|\cX_{j,n}\|_{L^p}^p}{\left(\sum_{j=1}^{u_n}a_j^2\right)^{p/2}}=0.
\end{equation}
Now, using \eqref{MaxMomAss} and that $\|X_{j,n}\|_{L^2}\leq Ca_j$ we have 
$$
\|\cX_{j,n}\|_{L^p}^p\leq A_p(w_n^p+a_j^p)
$$
for some constant $A_p>0$. Since $w_n=o(\sig_n^p)$ and $\left(\sum_{j=1}^{u_n}a_j^2\right)^{p/2}\asymp \sig_n^p$ it is enough to prove the following result.

\begin{lemma}\label{Lem}
Let $(a_j)$ diverge exponentially fast in a good way and let $p>2$. Then there is a constant $A=A_p>0$ so that for all $u$ large enough we have
$$
\frac{\sum_{j=1}^ua_j^p}{\left(\sum_{j=1}^{u}a_j^2\right)^{p/2}}\leq A(H(u))^{p/2-1}.
$$
\end{lemma}
The proof of Lemma \ref{Lem} is elementary, and it uses that for every $r>0$ we have
$$
\sum_{j=1}^ua_j^r\asymp \int_{(1-\del)u}^{u}e^{r(G(x))}dx= \int_{(1-\del)u}^{u}\left(e^{r(G(x))}G'(x)\right)/G'(x)dx\asymp H(u_n)^{-1}e^{rG(u_n)}
$$
where we have taken into account the properties of $G$ and $H$ described in Definition \ref{Def1}.
The lemma follows now by
taking $r=2$ and $r=p$, and the proof of Theorem \ref{Cor2.7} is complete.

\subsection{Proof of Corollary \ref{Cor2.9}}
Recall first the following maximal moment inequality.
If $\phi_n(m_n)<\frac12-\ve$ then by 
\cite[Theorem 6.17]{PelBook},  for every $p>2$ and $k,m$ such that $k+m\leq n$,
\begin{equation}\label{Pel}
\left\|\max\left\{\left|\sum_{k\leq j<k+l}\xi_{j,n}\right|:  0\leq l\leq m\right\}\right\|_{L^p}\leq C_{\ve}p\left(m_n\left\|\max_{k\leq j<k+m}|\xi_{j,n}|\right\|_{L^p}+\max_{0<l\leq m}\left\|\sum_{k\leq j<k+l}\xi_{j,n}\right\|_{L^2}\right).
\end{equation}

We thus see that \eqref{MaxMomAss} holds true under the conditions of item (i) in Corollary \ref{Cor2.9}. To show that is holds true under the conditions of item (ii) in Corollary \ref{Cor2.9} we notice that
$$
\left\|\max_{k\leq j<k+m}|\xi_{j,n}|\right\|_{L^p}^p\leq\sum_{j=1}^{n}\|\xi_{j,n}\|_{L^p}^p\leq n\max_{1\leq j\leq n}\|\xi_{j,n}\|_{L^p}^p.
$$

\subsection{ $\al$-mixing arrays}
The proof of Theorem \ref{Alpha} proceeds similarly to the proof of Theorem \ref{Rho}, by applying \cite[Corollary 2.2]{MP} and the results in Section \ref{Al Sec}. Since all the sufficient conditions (described above) for the maximal Lindeberg condition to hold true do not require any assumptions on the $\rho$-mixing coefficients, they also yield explicit conditions for the functional CLT for $\al$-mixing triangular arrays. In particular, Corollary \ref{al cor} holds true.


\begin{thebibliography}{Bow75}

\bibliographystyle{alpha}
\itemsep=\smallskipamount




\bibitem{Aldous}
D. Aldous, {\em Stopping times and tightness}, Ann. Probab. 6, (1978), 335-340.





\bibitem{Br1980}
R.C. Bradley, {\em A Remark on the Central Limit Question for Dependent Random Variables}, 
J.l of App. Prob. 17 (1980).


\bibitem{Br1987}
R.C. Bradley, {\em The Central Limit Theorem under $\rho$-mixing}, The Rocky Mountain Journal of Mathematics, 17 (1987).


\bibitem{Br}
R.C. Bradley, {\em Introduction to Strong Mixing Conditions}, Volume 1, Kendrick Press, Heber City, 2007.



\bibitem{Brown}
B. M. Brown {\em Martingale central limit theorem}, Ann. Math. Stat. 42, 59-66 (1971).













\bibitem{Dub}
R. Dobrushin, R. {\em Central limit theorems for non-stationary Markov chains I, II}. Theory Probab. Appl.1, 65-80, 329-383 (1956).

 



\bibitem{NonStatCLT}
M.Ekstr\"om, {\em A general central limit theorem for strong mixing sequences}, Stat. Prob. Lett. 94, 236-238 (2014).

\bibitem{Folland}
 G. B. Folland. {\em Real analysis. Pure and Applied Mathematics} (New York). John Wiley \& Sons, Inc., New York, second edition, 1999.






\bibitem{Gordin}
M.I. Gordin, {\em The central limit theorem for stationary processes}, Soviet. Math. Dokl.
10, 1174-1176  (1969).







\bibitem{IF}
M. Iosifescu and R. Theodorescu, {\em Random processes and learning}, Die Grundlehren der mathematischen Wissenschaften, Band 150. Springer-Verlag, New York (1969).









\bibitem{PelBook}
F. Merlev\'ede,  M. Peligrad, M. and S. Utev, S, {\em Functional Gaussian Approximation for Dependent Structures}, Oxford University Press (2019).




\bibitem{MPU}
F. Merlev\'ede,  M. Peligrad, M. and S. Utev, S, {\em Functional CLT for martingale-like nonstationary dependent structures}, Bernoulli 25(4B): 3203-3233 (November 2019)

\bibitem{MP}
F. Merlev\'ede and  M. Peligrad {\em Functional CLT for nonstationary strongly mixing processes}, Stat. Prob. Let. 156 (2020).




\bibitem{Peligrad1996}
M. Peligrad, {\em On the asymptotic normality of sequences of weak dependent random variables}. J. Theor. Prob. 9, 703--715 (1996).





\bibitem{Pel}
M. Peligrad, {\em Central limit theorem for triangular arrays
of non-homogeneous Markov chains},
Probab. Theory Relat. Fields (2012) 154:409-428.









\bibitem{Rio1995}
E. Rio, {\em About the Lindeberg method for strongly mixing sequences}, ESIAM: Probability and Statistics 1 35-61 (1995).

\bibitem{Ri} 
E. Rio, {\em Sur le th\' eor\` eme de Berry-Esseen pour les suites faiblement
d\' pendantes}, Probab. Th. Relat. Fields 104 (1996), 255-282.




\bibitem{RioBook}
E. Rio {\em Asymptotic Theory of Weakly Dependent Random Processes}, Probability
Theory and Stochastic Modelling 80, Springer (2017).

\bibitem{Rosen1956}
M. Rosenblat, {\em A central limit theorem and a strong mixing condition}, Proc. Natl. Acad. Set. 42, 43-47.








 
\bibitem{Shao1995}
Q.M. Shao, {\em Maximal inequalities for partial sums of $\rho$-mixing sequences}, Ann. Prob. 23,  948--965 (1995).








\bibitem{Utev1991}
S. Utev, {\em Sums of random variables with $\phi$-mixing}, [translation of Trudy Inst. Mat.
(Novosibirsk) 13 (1989), Asimptot. Analiz Raspred. Sluch. Protsess., 78-100]. Siberian Adv.
Math. 1, 24-155 (1991). 


\bibitem{VarSeth}
S. Sethuraman and S.R.S Varadhan, {\em A martingale proof of Dobrushin’s theorem for non-homogeneous
Markov chains}, Electron. J. Probab. 10, 1221–1235 (2005).

\bibitem{Volny}
D. Voln\'y, {\em A central limit theorem for non stationary mixing processes},  Commentationes Mathematicae Universitatis Carolinae, vol. 30 (1989), issue 2, pp. 405-407.




\end{thebibliography}
\end{document}